\documentclass[11pt,twoside,reqno,centertags]{amsart}
\usepackage{amsmath,amsthm,amsfonts,amssymb}
\advance\textheight by 10truemm
\newtheorem{Theorem}{Theorem}

\newtheorem{theorem}[Theorem]{Theorem}
\newtheorem{corollary}[Theorem]{Corollary}
\newtheorem{proposition}[Theorem]{Proposition}

\newtheorem{remark}[Theorem]{Remark}
\newcommand{\R}{{\mathbb R}}
\newcommand{\eps}{\varepsilon}
\font\basic=cmr10
\setbox1=\hbox{\basic\strut Department of Applied Mathematics and Statistics, Comenius University} 
\setbox2=\hbox{\basic\strut Mlynsk\'a dolina, 84248 Bratislava, Slovakia}
\setbox3=\hbox{\basic\strut email: {\tt quittner@fmph.uniba.sk}}

\begin{document}

\title[Liouville theorem and a priori estimates]{Liouville theorem \\ 
  and a priori estimates of radial solutions \\ for a non-cooperative elliptic system} 
\author{Pavol Quittner \\ \\ \box1 \\ \box2 \\ \box3}
\thanks{Supported in part by the Slovak Research and Development Agency 
        under the contract No. APVV-18-0308 and by VEGA grant 1/0339/21.} 
\date{}

\begin{abstract}
Liouville theorems for scaling invariant nonlinear elliptic
systems (saying that the system
does not possess nontrivial entire solutions)  
guarantee a priori estimates of solutions of
related, more general systems.
Assume that $p=2q+3>1$ is Sobolev subritical, $n\le3$ and $\beta\in\R$.
We first prove a Liouville theorem for the system
$$\left.\begin{aligned}
-\Delta u &=|u|^{2q+2}u+\beta|v|^{q+2}|u|^q u, \\
-\Delta v &=|v|^{2q+2}v+\beta|u|^{q+2}|v|^q v,
\end{aligned}\ \right\}
\quad\hbox{in}\quad \R^n,$$
in the class of radial functions $(u,v)$ 
such that the number of nodal domains of $u,v,u-v,u+v$ is finite.
Then we use this theorem to obtain a priori estimates
of solutions to related elliptic systems.
In the cubic case $q=0$, those solutions correspond 
to the solitary waves of a system of Schr\"odinger equations,
and their existence and multiplicity
have been intensively studied by various methods.
One of those methods is based on a priori estimates
of suitable global solutions of corresponding parabolic systems.
Unlike the previous studies, our Liouville theorem
yields those estimates for all $q\geq0$ which are Sobolev subcritical.

\vskip1cm
\medskip
\noindent \textbf{Keywords.} Liouville theorem, a priori estimate, elliptic system, Schr\"odinger equation

\medskip
\noindent \textbf{AMS Classification.}  35J10, 35J47, 35J61, 35B08, 35B45, 35B53, 35K58
\end{abstract}

\maketitle

\vfill\eject

\section{Introduction and main results}
\label{intro}

We are mainly interested in a priori estimates of radial solutions of the
problem
\begin{equation} \label{Fstat}
\begin{aligned}
&\left.\begin{aligned}
-\Delta u+\lambda u+\gamma v &=|u|^{2q+2}u+\beta|v|^{q+2}|u|^q u, \\
-\Delta v+\lambda v+\gamma u &=|v|^{2q+2}v+\beta|u|^{q+2}|v|^q v,
\end{aligned}\ \right\}
\quad\hbox{in}\quad \Omega, \\
&\ u=v=0 \hbox{ on }\partial\Omega\ \ \hbox{ if }\ \partial\Omega\ne\emptyset,
\end{aligned}
\end{equation}
where either $\Omega=B_R:=\{x\in\R^n:|x|<R\}$ or $\Omega=\R^n$, $n\leq3$, 
$\lambda,\gamma,\beta\in\R$,
$p:=2q+3\in(1,p_S)$,
and $p_S$ denotes the critical Sobolev exponent:
$$
p_S:=\left\{
  \begin{aligned}
    &\frac{n+2}{n-2},&\quad &\text{if $n\ge 3$,}\\
    &\infty,&\quad &\text{if $n\in\{1,2\}$.}
  \end{aligned}\right.
$$
In the cubic case $p=3$, solutions of \eqref{Fstat} correspond 
to the solitary waves of a system of Schr\"odinger equations 
and their existence and multiplicity
have been intensively studied by various (mainly variational) methods;
see the references in \cite{LW} or \cite{DTZ,ZW} if $\gamma=0$ or $\gamma\ne0$,
respectively.
The case $p\ne3$ has also been studied, see \cite{CS,CP} and the references
therein.

Topological and global bifurcation arguments often require
a priori estimates of solutions and such estimates have been obtained
for $n\leq3$, $p=3$ and positive solutions in \cite{BDW,DWW,DTZ}, for example,  
by proving and/or using suitable Liouville theorems for the related scaling invariant
problem
\begin{equation} \label{FstatL}
\left.\begin{aligned}
-\Delta u &=|u|^{2q+2}u+\beta|v|^{q+2}|u|^q u, \\
-\Delta v &=|v|^{2q+2}v+\beta|u|^{q+2}|v|^q v,
\end{aligned}\ \right\}
\quad\hbox{in}\quad \R^n,
\end{equation}
and the corresponding Dirichlet problem in a halfspace.

Another method of proving existence and multiplicity results for \eqref{Fstat}
is to consider the corresponding parabolic problem
\begin{equation} \label{Schr-lambda}
\begin{aligned}
&\left.
\begin{aligned}
u_t-\Delta u+\lambda u+\gamma v &=|u|^{2q+2}u+\beta|v|^{q+2}|u|^qu, \\
v_t-\Delta v+\lambda v+\gamma u &=|v|^{2q+2}v+\beta|u|^{q+2}|v|^qv,
\end{aligned}
\ \right\}
\ \hbox{ in } \Omega\times(0,\infty), \\
& u=v=0 \quad\hbox{ on }\partial\Omega\times(0,\infty)\ \ 
\hbox{ if }\ \partial\Omega\ne\emptyset, 
\end{aligned}
\end{equation}
and use the fact that (if $\gamma=0$, then)
the number of zeroes and intersections of radial solutions of
\eqref{Schr-lambda} is nonincreasing in time.
Such arguments
have been used in \cite{WW,LW} if $n\le3$, $p=3$, $\lambda>0=\gamma$,
and they again require a priori estimates
of suitable global solutions of \eqref{Schr-lambda}.

The arguments in the proofs of a priori estimates
in \cite{BDW,DWW,DTZ} or \cite{WW,LW} do not allow one to
cover the full subcritical range $p<p_S$ if $n=3$
or $2\leq n\leq3$, respectively
(see Remark~\ref{remTech} for more details).
The main result of this paper is a Liouville theorem
for radial solutions of \eqref{FstatL},
(possibly nonradial) solutions of  \eqref{FstatL} with $n=1$, 
and solutions of the problem
\begin{equation} \label{FstatL-half}
\begin{aligned}
&\left.
\begin{aligned}
-u_{xx}&= |u|^{2q+2}u+\beta|v|^{q+2}|u|^qu, \\
-v_{xx}&= |v|^{2q+2}v+\beta|u|^{q+2}|v|^qv,
\end{aligned}
\ \right\}
\ \hbox{ in } (0,\infty), \\
&\quad u(0)=v(0)=0
\end{aligned}
\end{equation}
(see Theorem~\ref{thmSchrL}).
Using that theorem we obtain the required a priori estimates
(for both \eqref{Fstat} and \eqref{Schr-lambda})
in the full subcritical range. 
In the case of \eqref{Schr-lambda} we will also assume
$p\geq3$ (i.e.~$q\geq0$) in order to avoid
some technical problems with local existence
and uniqueness of solutions (if $q<0$, then 
the nonlinearity in \eqref{Schr-lambda} is not Lipschitz continuous). 

To formulate our results more precisely, 
let us introduce some notation first.
By a nontrivial solution we understand
a solution $(u,v)$ such that $(u,v)\not\equiv(0,0)$.

If $J\subset \R$ is an interval and $v\in C(J,\R)$,
then we define 
$$ \begin{aligned}
  z(v)=z_J(v):=\sup\{j:&\,\exists x_1,\dots,x_{j+1}\in J,\
     x_1<x_2<\dots<x_{j+1},\\ 
    &v(x_i)\cdot v(x_{i+1})<0 \hbox{ for } i=1,2,\dots,j\},
\end{aligned}
$$
where $\sup(\emptyset):=0$. We usually refer to $z_J(v)$  as the \emph{zero
number} of $v$ in $J$. Note that $z_J(v)$ is actually the number of
sign changes of $v$; it coincides with the number of zeros of $v$ if 
$J$ is open, $v\in C^1(J)$ and all its zeros are simple. 
If $v:\R^n\to\R$ is a continuous, radially symmetric function,
i.e. $v(x)=\tilde v(|x|)$ for some $\tilde v\in C([0,\infty),\R)$,
then we define $z(v):=z(\tilde v)$.
Given $C_1,C_2,C_3,C_4\geq0$, set 
$$ 
\begin{aligned}
{\mathcal K}={\mathcal K}(C_1,C_2,C_3,C_4)&:=\{(u,v): z(u)\leq C_1,\ z(v)\leq C_2,\ 
          z(u-v)\leq C_3,\ z(u+v)\leq C_4\},\\
{\mathcal K^+}={\mathcal K^+}(C_3)&:=\{(u,v): u,v\geq0,\  z(u-v)\leq C_3\},\\
{\mathcal K}^*&:=\{(u,v)\in{\mathcal K}: u\not\equiv\pm v\},
\end{aligned}
$$ 
and notice that ${\mathcal K^+}\subset{\mathcal K}(0,0,C_3,0)$.

The following Liouville theorem has already been proved in \cite{BDW} 
in the case of nonnegative solutions, $\beta<1$ and $p=3$.
Notice also that if $\beta\in(-1,\infty)$ or $\beta>0$ and one considers nonnegative solutions, 
then the nonexistence of nontrivial (radial and nonradial) solutions 
to problems occurring in the following theorem has been studied in
\cite{QS9,DW} or \cite{RZ}, respectively.

\begin{theorem} \label{thmSchrL}
Assume $n\leq3$ and $p=2q+3\in(1,p_S)$. Let $C_1,C_2,C_3,C_4\geq0$ be fixed.
If $\beta\ne-1$, then system \eqref{FstatL} does not possess 
nontrivial classical radial solutions satisfying $(u,v)\in{\mathcal K}$ 
and system \eqref{FstatL} with $n=1$ does not possess 
nontrivial classical solutions satisfying $(u,v)\in{\mathcal K}$.
If $\beta=-1$, then 
all classical radial solutions of \eqref{FstatL} satisfying $(u,v)\in{\mathcal K}$ 
and all classical solutions of system \eqref{FstatL} with $n=1$
satisfying $(u,v)\in{\mathcal K}$
are of the form $(c,\pm c)$, where $c\in\R$. 
Problem \eqref{FstatL-half} does not possess 
nontrivial classical solutions satisfying $(u,v)\in{\mathcal K}$ for any $\beta\in\R$.
\end{theorem} 

Theorem~\ref{thmSchrL} combined with scaling and doubling arguments from
\cite{PQS1}, and an argument due to \cite{B12} (based on the Sturm
comparison theorem) yield the following result:

\begin{theorem} \label{thmUB}
Assume $\Omega=\R^n$ or $\Omega=B_R$, $n\leq3$, $\lambda,\gamma\in\R$ and $p=2q+3\in(1,p_S)$.
Let $C_1,C_2,C_3,C_4\geq0$ be fixed. Let $B$ be a compact set in
$\R\setminus\{-1\}$ and $B^*$ be a compact set in $\R$.
Then there exists $C$ such that any classical radial solution $(u,v)\in{\mathcal K}$ of \eqref{Fstat}
with $\beta\in B$, and any classical radial solution $(u,v)\in{\mathcal K}^*$ of \eqref{Fstat} with $\beta\in B^*$ 
satisfies $\|(u,v)\|_\infty\leq C$.
\end{theorem}

The proof of Theorem~\ref{thmUB} shows that this theorem remains true for solutions
of large classes of systems which are perturbations
of the scaling invariant system \eqref{FstatL}.
In particular, the estimate $\|(u,v)\|_\infty\leq C$ in Theorem~\ref{thmUB}
is locally uniform with respect to $\lambda$ and $\gamma$.

A straightforward modification of the proof of Theorem~\ref{thmUB}
(cf.~\cite{PQS1}) also guarantees universal singularity estimates.
More precisely, if $\Omega:=B_R\setminus\{0\}$, $R>2$, and $p,\lambda,\gamma,B,B^*$ are as in Theorem~\ref{thmUB},
then there exists $C>0$ such that any classical radial solution $(u,v)\in{\mathcal K}$ of
the system of PDEs in \eqref{Fstat} with $\beta\in B$, 
and any classical radial solution $(u,v)\in{\mathcal K}^*$ of the system of
PDES in \eqref{Fstat} with $\beta\in B^*$
satisfies the estimate
$$|u(x)|+|v(x)|\leq C|x|^{-2/(p-1)}, \qquad0<|x|<1.$$
(The solution $(u,v)$ need not satisfy the boundary condition in \eqref{Fstat}.)

Theorem~\ref{thmSchrL} and \cite{Q-LSP} 
guarantee that the related scaling invariant parabolic problem
\begin{equation} \label{Schr-Liouv}
\left.
\begin{aligned}
u_t-\Delta u &=|u|^{2q+2}u+\beta|v|^{q+2}|u|^qu, \\
v_t-\Delta v &=|v|^{2q+2}v+\beta|u|^{q+2}|v|^qv,
\end{aligned}
\ \right\}
\ \hbox{ in } \R^n\times\R,
\end{equation}
does not possess nontrivial radial solutions
satisfying $(u,v)(\cdot,t)\in{\mathcal K}$ for all $t\in\R$,
and problems \eqref{Schr-Liouv} with $n=1$ and
\begin{equation} \label{Schr-bdry}
\begin{aligned}
&\left.
\begin{aligned}
u_t-u_{xx} &= |u|^{2q+2}u+\beta|v|^{q+2}|u|^qu, \\
v_t-v_{xx} &= |v|^{2q+2}v+\beta|u|^{q+2}|v|^qv,
\end{aligned}
\ \right\}
\ \hbox{ in } (0,\infty)\times\R, \\
&\quad u=v=0 \ \hbox{ on } \{0\}\times\R,
\end{aligned}
\end{equation}
do not possess nontrivial solutions
satisfying $(u,v)(\cdot,t)\in{\mathcal K}$ for all $t\in\R$.
These parabolic Liouville theorems together with scaling and doubling arguments in \cite{PQS}
immediately imply the following universal $L^\infty$-estimate for global solutions of \eqref{Schr-lambda}
(see \cite[Corollary~5]{PQS} for a more general statement):  

\begin{corollary} \label{corS}
Assume $\Omega=\R^n$ or $\Omega=B_R$, $n\leq3$, 
$\beta\ne-1$ and $p=2q+3\in(1,p_S)$. 
Then there exists $C>0$ such that any global radial classical solution of \eqref{Schr-lambda}
with $(u,v)(\cdot,t)\in{\mathcal K}$
for all $t\in(0,\infty)$ satisfies the following estimate:
$$ \|(u,v)(\cdot,t)\|_\infty\leq C(1+t^{-1/(p-1)}),\quad t\in(0,\infty).$$  
\end{corollary}

The constant $C=C(\beta,\lambda,\gamma)$ in Corollary~\ref{corS} is locally uniform for 
$\beta\in\R\setminus\{0\}$ and $\lambda,\gamma\in\R$.
Notice also that ${\mathcal K}$ or ${\mathcal K}^+$ is invariant
with respect to the semiflow generated by \eqref{Schr-lambda} if $\gamma=0$
or $\gamma\leq0$, respectively. 

Corollary~\ref{corS} can be used to prove
the following uniform $H^1$-estimate for global radial solutions 
of \eqref{Schr-lambda} with bounded energy and initial data in $H^1\cap{\mathcal K}$ or $H^1\cap{\mathcal K}^+$.
By $H^1_r$ we denote the set of radial functions in $H^1$
and by $\|\cdot\|$ the norm in $H^1(\Omega,\R^2)$.
We also set
\begin{equation} \label{FG}
\begin{aligned}
 {\mathcal U} &:=(u,v), \\
{\mathcal F}({\mathcal U}) &:=(|u|^{2q+2}u+\beta|v|^{q+2}|u|^qu,|v|^{2q+2}v+\beta|u|^{q+2}|v|^qv), \\
G({\mathcal U}) &:=\frac1{p+1}{\mathcal F}({\mathcal U})\cdot {\mathcal U} \quad\hbox{ (hence $\nabla G={\mathcal F}$)}.
\end{aligned}
\end{equation}

\begin{proposition} \label{propCL}
Assume $\Omega=\R^n$ or $\Omega=B_R$, $n\leq3$, $\beta\ne-1$, $\lambda>0\geq\gamma$, $p=2q+3\in[3,p_S)$.
If $\Omega=\R^n$, then assume also $\lambda+\gamma>0$. 
Let ${\mathcal U}_0\in H^1_r(\Omega,\R^2)$.
If $\gamma=0$ or $\gamma<0$, then assume also
${\mathcal U}_0\in{\mathcal K}$ or ${\mathcal U}_0\in{\mathcal K}^+$,
respectively.
Assume that the solution of \eqref{Schr-lambda}
with initial data ${\mathcal U}(\cdot,0)={\mathcal U}_0$ is global
and satisfies $|E(t)|\leq C_E$ for $t>0$, where
$$E(t):=\frac12\int_\Omega(|\nabla {\mathcal U}(x,t)|^2+\lambda |{\mathcal U}(x,t)|^2)\,dx
 +\gamma\int_\Omega(uv)(x,t)\,dx
-\int_\Omega G({\mathcal U}(x,t))\,dx.$$
Then
\begin{equation} \label{estCL}
 \|{\mathcal U}(\cdot,t)\|\leq C=C(\|{\mathcal U}_0\|,C_E).
\end{equation}
\end{proposition}

The $H^1$-estimate in Proposition~\ref{propCL} is based on the universal $L^\infty$-estimates
in Corollary~\ref{corS}, but the universality of those estimates is not needed: 
It would be sufficient to use $L^\infty$-estimates which can depend on $\|{\mathcal U}_0\|$ and $C_E$,
and such estimates could likely be obtained directly from the elliptic Liouville theorem (Theorem~\ref{thmSchrL})
by using the approach in \cite{G} (hence we would not need the parabolic Liouville theorems in \cite{Q-LSP}).
On the other hand, universal $L^\infty$-estimates
as in Corollary~\ref{corS} also enable one to prove the existence of periodic solutions
of related problems with time-periodic coefficients, for example (see \cite[Section~6]{BPQ}),
and such results cannot be obtained by using the weaker estimates depending
on $\|{\mathcal U}_0\|$ and $C_E$.

As already mentioned,  the authors of \cite{LW,WW} use the 
properties of the parabolic semiflow in order to prove
the existence and multiplicity of nontrivial radial solutions
of \eqref{Fstat} with $n\leq3$, $\lambda>0$ and $q=\gamma=0$.
More precisely, paper \cite{WW} deals with positive radial solutions,
$\Omega=B_R$ and $\beta\leq-1$,
and paper \cite{LW} with nodal radial solutions of various generalizations of \eqref{Fstat}
and $\beta<0$ (or $\beta<\beta_0$, where $\beta_0>0$ is small enough).
In both papers, a priori estimates of suitable global solutions of \eqref{Schr-lambda} play an important role.
If we consider initial data $U_0\in{\mathcal A}$, where
${\mathcal A}$ is the domain of attraction of the zero solution,
then the solution of \eqref{Schr-lambda} is global and the corresponding energy function $E(t)$ is bounded,
hence estimate \eqref{estCL} is true (provided the remaining assumptions
in Proposition~\ref{propCL} are satisfied).
Estimate \eqref{estCL} then also guarantees that the solutions of \eqref{Schr-lambda} 
with initial data $U_0\in\partial{\mathcal A}$ are global
and satisfy \eqref{estCL}, and these particular global solutions
are used in \cite{LW,WW} in order to find solutions of \eqref{Fstat}
with prescribed number of nodal domains or intersections.
The arguments in \cite{WW} also require some compactness of those particular global solutions, 
and such compactness is guaranteed by the next proposition.
 
\begin{proposition} \label{prop-comp}
Let the assumptions of Proposition~\ref{propCL} be satisfied.
If $\Omega=\R^n$, then assume also that ${\mathcal U}_0$ is compactly supported and $n\geq2$.
Then the trajectory $t\in[0,\infty)\to H^1_r(\Omega,\R^2):t\mapsto {\mathcal U}(\cdot,t)$ is compact.
\end{proposition}

The proof in \cite{WW} guaranteeing the existence of
positive solutions of \eqref{Fstat} with prescribed number of intersections
required $\Omega=B_R$, $p=3$, and the authors of \cite{WW} also assume $\gamma=0$. 
Propositions~\ref{propCL} and \ref{prop-comp} enable one to
prove analogous results also for $\Omega=\R^n$ and $p\in[3,p_S)$.
In addition, one can also consider the case $\gamma<0$: If $\Omega=B_R$,
then in order to guarantee the stability of the zero solution, 
one has to assume $\lambda+\gamma>-\lambda_1$,
where $\lambda_1$ is the first eigenvalue of the negative Dirichlet
Laplacian in $\Omega$.

Similarly, Proposition~\ref{propCL} indicates
that many arguments from \cite{LW} guaranteeing the existence
of solutions of \eqref{Fstat} with prescribed number
of nodal domains in the cubic case $p=3$
can also be used if $p\in(3,p_S)$.

\begin{remark} \label{remTech} \rm
(i) The proofs of Liouville theorems used in \cite{BDW,DWW,DTZ} heavily
depend on the choice $p=3$: 
The arguments in those proofs cannot be used if $n=3$ and $p>3$, for example.

(ii) The bounds of global solutions of  \eqref{Schr-lambda}
in \cite{WW,LW} are proved by integral estimates (cf.~\cite{CL})
which require $p:=2q+3<p_{CL}:=(3n+8)/(3n-4)$.
Condition $p<p_{CL}$ can likely be improved to $p<p_S$ by a bootstrap argument due to
\cite{Q-AMUC} (see also \cite{Q-HJM} or \cite{GMS,HZ} for applications of this
argument to more general or rescaled problems), but only if $\beta>-1$. 
If $\beta\leq-1$, then a modification
of that bootstrap argument could likely improve the condition
$p<p_{CL}$ slightly if $n=3$ (to $p<p_{CL}+1/5$), but not for $n=2$, cf.~\cite{ABKQ}.
Our results guarantee that the required a priori estimates
remain true for any $p\in[3,p_S)$ if $n\leq3$ and $\beta\ne-1$. 

On the other hand, if $p=3$, $n\le3$, $k$ is a fixed positive integer, $\lambda>0\geq\gamma$ and $\lambda+\gamma>0$,
then the integral estimates in \cite{WW} of suitable global positive solutions $(u,v)$ of \eqref{Schr-lambda} 
satisfying $z(u-v)\leq k$  are locally uniform for $\beta\in\R$.
(In fact, \cite{WW} deals with $\gamma=0$, $\beta\leq-1$ and $\Omega=B_R$ only, 
but these assumptions are not needed for such estimates.)  
Assume that the $\omega$-limit set of such global solution $(u,v)$ contains a positive stationary
solution of the form $(u^*,u^*)$. Since the norms of such positive stationary
solutions tend  to $\infty$ as $\beta\to-1$, this would yield a contradiction if $\beta$ is close to $-1$.
Consequently, the topological arguments in the proof of \cite[Theorem 1.1]{WW} 
leading to the existence of stationary solutions satisfying $z(u-v)=k$
can be used whenever $\beta<-1+\eps_k$, where $\eps_k>0$ is small enough.
Our bounds based on Liouville theorems are locally uniform with respect to $\beta$
only for $\beta\in\R\setminus\{-1\}$, hence such arguments cannot be used. 
The reason is that we are using the universal estimates in Corollary~\ref{corS}
which are true for all solutions in ${\mathcal K}$ including solutions
of the form $(u,u)$, hence they cannot be uniform as $\beta$ approaches $-1$.
\end{remark}

\section{Proofs} 
\label{sec-Schr}

\begin{proof}[Proof of Theorem~\ref{thmSchrL}] 
Due to scaling and doubling arguments (see \cite{PQS1}), 
we only have to prove the nonexistence for bounded solutions. 
Assume that $(u,v)\in{\mathcal K}$ is a nontrivial bounded radial solution
of \eqref{FstatL}
and $(u,v)\ne(c,\pm c)$ if $\beta=-1$.
Consider $u,v$ as functions of the radial variable $r=|x|$,
$\Delta u(r)=u''(r)+\frac{n-1}ru'(r)$.
System \eqref{FstatL} possesses nontrivial radial solutions of the form
$W_0:=(w,\pm w)$ or $W_1:=(w,0)$ or $W_2:=(0,w)$,
where $z(w)<\infty$, only if $\beta=-1$, and such solutions are of the form $(c,\pm c)$ with $c\ne0$
(see \cite[Theorem 2.2]{Ni} in the case of $W_1,W_2$ or
$W_0$ and $\beta>-1$, and see \cite[Proposition~4]{Q-JDE} in the case of $W_0$ and $\beta<-1$),
hence we have $u\not\equiv v$, $u\not\equiv-v$, $u\not\equiv0$ and $v\not\equiv0$. 
Replacing $u$ by $-u$ and/or $v$ by $-v$ if necessary, 
we may assume that there exists $R_0\geq 0$ such that
\begin{equation} \label{uv0}
u(r)>v(r)>0 \ \hbox{ for }\ r>R_0.
\end{equation} 

Assume first $n\leq2$ or $n=3$ and $p=2q+3\leq3$.
Set $w:=u-v$ if $\beta\leq0$, and $w:=u$ otherwise.
If $r>R_0$, then $w(r)>0$ and $-\Delta w \geq w^p$, which
contradicts the corresponding Liouville-type theorem for inequalities
in exterior domains, see \cite{BVP}, for example.
The same argument applies to (possibly nonradial) solutions of \eqref{FstatL} in $\R^1$,
and to solutions of  \eqref{FstatL-half}. 
Consequently, we just have to prove the nonexistence of bounded radial solutions of \eqref{FstatL}
satisfying \eqref{uv0} in the case $n=3$ and $p=2q+3\in(3,5)$.

Theorem~\ref{thmSchrL} for $n=1$ (which we have just proved)
together with scaling and doubling arguments (see \cite{BPQ}, for example)
imply 
\begin{equation} \label{decay}
 |u(r)|+|v(r)|+r(|u'(r)|+|v'(r)|)\leq C^*r^{-2/(p-1)},\quad r>0.
\end{equation}
If $\beta>-1$, then 
\begin{equation} \label{Gp}
C_1|{\mathcal U}|^{p+1}\leq G({\mathcal U})\leq C_2|{\mathcal U}|^{p+1} \ \hbox{ for any }\ {\mathcal U}=(u,v),
\end{equation}
where $G$ is defined in \eqref{FG}.
In addition, the Rellich-Pohozaev identity \cite[Lemma~3.6]{QS9}
(which is true also for nodal solutions) implies
\begin{equation} \label{RP}
\int_0^R c_pG({\mathcal U}(r))r^2\,dr = R^3\bigl(2G({\mathcal U}(R))+|{\mathcal U}'(R)|^2+\frac1R {\mathcal U}(R)\cdot {\mathcal U}'(R)\bigr),
\end{equation}
where $c_p:=5-p>0$. 
Now \eqref{RP}, \eqref{Gp} and \eqref{decay} imply
$$\int_0^R|{\mathcal U}(r)|^{p+1}r^2\,dr\leq CR^{-\frac{5-p}{p-1}} \ \to\ 0\ \hbox{ as }\ R\to\infty,$$
which yields a contradiction.

It remains to consider the case $n=3$, $p=2q+3\in(3,5)$ and $\beta\le-1$.
Our arguments in this case are inspired by the proof of \cite[Theorem~2.5]{Ni}. 
In the rest of the proof we denote 
$U(r):=r^{2/(p-1)}u(r)$, $V(r):=r^{2/(p-1)}v(r)$. 
Then \eqref{decay} guarantees
\begin{equation} \label{Cstar}
 |U(r)|+|V(r)|\leq C^*,\quad r|U'(r)|+r|V'(r)|\leq 2C^*, \quad r>0.
\end{equation}
If $Z\in\{U,V\}$, then $Z$ solves the equation
\begin{equation} \label{Z}
r^2Z''+arZ'-bZ+F(Z)=0, 
\end{equation}
where 
$$ a=\frac{2(p-3)}{p-1}\in(0,1),\qquad
  b=\frac{2(p-3)}{(p-1)^2}\in(0,\frac14), $$
and 
$$F(Z)=\begin{cases}
      |U|^{p-1}U+\beta|V|^{q+2}|U|^qU &\hbox{ if $Z=U$},\\
      |V|^{p-1}V+\beta|U|^{q+2}|V|^qV &\hbox{ if $Z=V$}.
 \end{cases}
$$
Set also
$$ E:=-\frac b2(U^2+V^2)+\frac1{p+1}(|U|^{p+1}+|V|^{p+1})+\frac{2\beta}{p+1}|UV|^{(p+1)/2},$$
$$ \varphi:=(U')^2+(V')^2.$$
Multiplying \eqref{Z} with $Z=U$ or $Z=V$ by $U'$ or $V'$, respectively,
and adding the resulting equations we obtain
\begin{equation} \label{Ephi}
 \frac12 r^2\varphi'(r)+ar\varphi(r)+E'(r)=0,
\end{equation}
and integration by parts yields
\begin{equation} \label{Ephiint}
 \frac12\bigl(\rho^2\varphi(\rho)-r^2\varphi(r)\bigr)-(1-a)\int_{r}^{\rho}s\varphi(s)\,ds+E(\rho)-E(r)=0,\quad \rho>r.
\end{equation}

If $r>R_0$, then \eqref{uv0} and $\beta\le-1$ imply $F(V(r))\leq 0$.
Assume
\begin{equation} \label{Vincr}
 V'(r_0)\geq0\ \hbox{ for some }\ r_0>R_0. 
\end{equation}
Then $V'>0$ on $(r_0,\infty)$, since  $V''>0$ whenever $V'=0$.
Fix $r_1>r_0$ and set 
$$\eps:=\min(bV(r_1),ar_1V'(r_1))>0.$$ 
If $ar_2V'(r_2)<\eps$ for some $r_2>r_1$,
then set $r_3:=\inf\{r<r_2:a\rho V'(\rho)<\eps\hbox{ on }[r,r_2]\}$ and notice that $r_3\in[r_1,r_2)$,
$ar_3V'(r_3)=\eps$
and $bV(r)>bV(r_1)\geq\eps$ for $r>r_1$. These estimates, \eqref{Z} and $F(V)\leq 0$
guarantee $V''>0$ on $(r_3,r_2)$, hence $ar_2V'(r_2)>ar_3V'(r_3)=\eps$ which yields a contradiction.
Consequently, $arV'(r)\geq\eps$ for $r>r_1$,
which contradicts the boundedness of $V$.
Thus \eqref{Vincr} fails and we have
$V'<0$ on $(R_0,\infty)$. 

If $V_\infty:=\lim_{r\to\infty}V(r)>0$, then \eqref{Z} implies $r^2V''(r)>bV_\infty/2=:c_V$  for $r>r_4$,
hence 
considering $R\to\infty$ in the estimate
$$
-V'(r)>V'(R)-V'(r)=\int_r^R V''(\rho)\,d\rho>c_V\int_r^R \frac1{\rho^2}\,d\rho=c_V\Bigl(\frac1r-\frac1R\Bigr)
$$
we obtain
 $V'(r)\leq-c_V/r$ for $r>r_4$, which contradicts the boundedness of $V$.
Thus $V_\infty=0$ and $q>0$ implies $F(V(r))=o(V(r))$ as $r\to\infty$.  
Consequently, there exists a positive nonincreasing function
$f$ such that $f(r)\to0$ as $r\to\infty$ and
$$r^2V''(r)+arV'(r)\in(0,f(r))\ \hbox{ for $r$ large}. $$
Assume $(1-a)rV'(r)<-f(r)$ for some $r$ large.
Then $r(rV'(r))'=r^2V''(r)+rV'(r)<r^2V''(r)+arV'(r)-f(r)<0$,
hence $(1-a)\rho V'(\rho)<-f(r)\leq-f(\rho)$ for $\rho>r$.
The inequality $|V'(\rho)|>\frac{f(r)}{1-a}\frac1\rho$ contradicts the boundedness of $V$.
Hence
\begin{equation} \label{rVr}
V(r)+r|V'(r)|=o(1) \ \hbox{ as }\ r\to\infty.
\end{equation}

Fix $M:=e^2$, $\eps_k\searrow0$ and choose $R_k\nearrow\infty$ such that $R_1>R_0$ and 
\begin{equation} \label{Vsmall}
V(r)+r|V'(r)|<\eps_k \quad\hbox{ for }r\geq R_k.
\end{equation}
We have two possibilities:

Case A: $(\forall k)\,(\exists r_k\geq R_k)\, 0<U\leq\eps_k$ on $[r_k,Mr_k]$.

Case B: $(\exists k_0)\,(\forall r\geq R_{k_0})\,(\exists \tilde r\in[r,Mr])\, U(\tilde r)>\eps_{k_0}$.

Consider Case A first. If $r^2\varphi(r)\geq2\eps_k^2$ on $J_k:=[r_k,Mr_k]$,
then \eqref{Vsmall} implies $r|U'(r)|\geq\eps_k$ on $J_k$, hence
$$ \eps_k\geq|U(Mr_k)-U(r_k)|=\Big|\int_{J_k}U'(r)\,dr\Big|\geq\int_{J_k}\frac{\eps_k}r\,dr=2\eps_k, $$
which yields a contradiction. Consequently, 
there exists $\tilde R_k\in J_k$ such that $\tilde R_k^2\varphi(\tilde R_k)<2\eps_k^2$.
Since $U(\tilde R_k),V(\tilde R_k)\to0$, we have
\begin{equation} \label{tildeA}
E(\tilde R_k)\to 0, \quad \tilde R_k^2\varphi(\tilde R_k)\to0,\quad \tilde R_k\to\infty.
\end{equation}

Next consider Case B. Set $\eps^*:=\eps_{k_0}$, $R^*:=R_{k_0}$,
$I_k:=[M^{k-1}R^*,M^kR^*]$, $k=1,2,\dots$.
For each $k$ there exists $\tilde r_k\in I_k$ such that $U(\tilde r_k)\in[\eps^*,C^*]$.
Set
$$u_k(\rho):=\tilde r_k^{2/(p-1)}u(\tilde r_k\rho),
  \quad v_k(\rho):=\tilde r_k^{2/(p-1)}v(\tilde r_k\rho),\quad \rho>R_0/\tilde r_k. $$
Then $u_k,v_k>0$ are locally bounded, $v_k\to0$ locally uniformly,
$u_k(1)=U(\tilde r_k)\in[\eps^*,C^*]$ and
$$ 0=\Delta u_k+u_k^p+\beta v_k^{q+2}u_k^{q+1}.$$
Consequently, a subsequence $u_{k_j}$ converges in $C_{loc}$ to a positive
solution $\tilde u$ of $\Delta u+u^p=0$ in $(0,\infty)$.
Fix $m\geq1$ and set $\rho_{j,m}:=\tilde r_{k_j+m}/\tilde r_{k_j}\in[M^{m-1},M^{m+1}]$.
Then 
$$u_{k_j}(\rho_{j,m})=\tilde r_{k_j}^{2/(p-1)}u(\tilde r_{k_j+m})
  = \rho_{j,m}^{-2/(p-1)}U(\tilde r_{k_j+m}) \geq \rho_{j,m}^{-2/(p-1)}\eps^*.$$
Since $u_{k_j}\rightrightarrows\tilde u$ on $[M^{m-1},M^{m+1}]$,
there exists $\rho_m\in[M^{m-1},M^{m+1}]$ such that $\tilde u(\rho_m)\geq \eps^*\rho_m^{-2/(p-1)}$.
Hence  $\limsup_{\rho\to\infty}\tilde u(\rho)\rho^{2/(p-1)}\geq\eps^*$ and
\cite[Remark 9.5]{SPP2} (see also \cite{GS,SZ}) 
shows that
$\tilde u(\rho)=b^{1/(p-1)}\rho^{-2/(p-1)}$.
Consequently, $U(\tilde r_{k_j}\rho)\to b^{1/(p-1)}$, $V(\tilde r_{k_j}\rho)\to 0$  and 
$E(\tilde r_{k_j}\rho)\to E_\infty:=-\frac{p-1}{2(p+1)}b^{(p+1)/(p-1)}$, locally uniformly with respect to $\rho>0$.
Fix $\eps\in(0,-E_\infty)$ and
$0<\rho_1<\rho_2$ such that $\log(\rho_2/\rho_1)>{2C^*\eps^{-1/2}}$, and set $J_j:=(\tilde r_{k_j}\rho_1,\tilde r_{k_j}\rho_2)$.
Assume that
\begin{equation} \label{limsup}
\limsup_{j\to\infty}\inf_{r\in J_j}r^2\varphi(r)\geq\eps.
\end{equation}
Then \eqref{rVr} implies 
$$\limsup_{j\to\infty}\inf_{r\in J_j}r|(U'(r)|\geq \sqrt{\eps/2},$$
hence for suitable $j$ large we obtain $r|U'(r)|\geq\sqrt\eps/2$ on $J_j$
and $|\int_{J_j}U'(r)\,dr|\geq \int_{J_j}\frac{\sqrt\eps}{2r}\,dr>C^*$,
which contradicts \eqref{Cstar}. 
Consequently, \eqref{limsup} fails, hence if $j$ is large, then there exists $\tilde R_j\in J_{j}$ such that 
\begin{equation} \label{tildeB}
\tilde R_j^2\varphi(\tilde R_j)<\eps<-E_\infty,\quad
E(\tilde R_j)\to E_\infty,\quad \tilde R_j\to\infty.
\end{equation}

Notice that $E(0)=0$ and $\lim_{r\to0+}r^2\varphi(r)=0$. In both Case 
A and B, due to \eqref{tildeA} and \eqref{tildeB}, respectively,
we can pass to the limit in \eqref{Ephiint} with $r:=0$ and 
$\rho:=\tilde R_k$ (or $\rho:=\tilde R_j$)
to obtain $\int_0^\infty s\varphi(s)\,ds\leq0$, which yields a contradiction.
\end{proof}

\begin{proof}[Proof of Theorem~\ref{thmUB}]
If $\Omega=\R^n$, then set $R:=\infty$. Radial solutions $(u,v)$ will be considered
as functions of $r:=|x|\in[0,R)$.
 
Assume to the contrary that there exist $\beta_k\in B$ and radial solutions
$(u_k,v_k)\in{\mathcal K}$ (or $\beta_k\in B^*$ and $(u_k,v_k)\in{\mathcal K}^*$)
such that $\|(u_k,v_k)\|_\infty\to\infty$.
Then there exist $r_k\in[0,R)$ such that
$M_k:=M(u_k,v_k)(r_k)\to\infty$, where
$$M(u,v):=|u|^{(p-1)/2}+|v|^{(p-1)/2}+|u'|^{(p-1)/(p+1)}+|v'|^{(p-1)/(p+1)}.$$
The Doubling Lemma in \cite{PQS1} guarantees that we may assume
$$ M(u_k,v_k)\leq 2M_k \ \hbox{ on }\ \{r\in[0,R):|r-r_k|\leq \frac k{M_k.}\}$$
Set $\lambda_k:=1/M_k$.
We may assume that $\beta_k\to\beta$ and also that one of the following three
cases occur:

Case A: $r_k/\lambda_k\to c_0\geq0$.

Case B: $r_k/\lambda_k\to\infty$ and either $R=\infty$ or $(R-r_k)/\lambda_k\to\infty$.

Case C: $R<\infty$ and $(R-r_k)/\lambda_k\to c_R\geq0$.

We set
$$ \tilde u_k(\rho):=\begin{cases}
                \lambda_k^{2/(p-1)}u_k(\lambda_k\rho) & \hbox{ in Case A},\\
                \lambda_k^{2/(p-1)}u_k(r_k+\lambda_k\rho) & \hbox{ in Case B},\\
                \lambda_k^{2/(p-1)}u_k(R-\lambda_k\rho) & \hbox{ in Case C},
        \end{cases} $$ 
and we define $\tilde v_k$ analogously.
We also set
$$ \rho_k:=\begin{cases}  
   0 &  \hbox{ in Case A},\\   r_k/\lambda_k & \hbox{ in Case B},\\  -R/\lambda_k & \hbox{ in Case C},
        \end{cases}
  \qquad \tilde\rho_k:=\begin{cases} 
  r_k/\lambda_k\to c_0  &  \hbox{ in Case A},\\ 0 &  \hbox{ in Case B},\\ (R-r_k)/\lambda_k\to c_R & \hbox{ in Case C}.
          \end{cases}$$
Then
$$ \begin{aligned} 
  \tilde u_k''+\frac{n-1}{\rho+\rho_k}\tilde u_k'-\lambda_k^2(\lambda\tilde u_k+\gamma\tilde v_k)
 &+ |\tilde u_k|^{p-1}\tilde u_k+\beta_k|\tilde v_k|^{q+2}|\tilde u_k|^q\tilde u_k=0, \\
  \tilde v_k''+\frac{n-1}{\rho+\rho_k}\tilde v_k'-\lambda_k^2(\lambda\tilde v_k+\gamma\tilde u_k)
 &+ |\tilde v_k|^{p-1}\tilde v_k+\beta_k|\tilde u_k|^{q+2}|\tilde v_k|^q\tilde v_k=0,
\end{aligned}$$
$M(\tilde u_k,\tilde v_k)(\tilde\rho_k)=1$, and
$M(\tilde u_k,\tilde v_k)(\rho)\leq2$ whenever
$$ \begin{cases} 
   \rho\in[0,R/\lambda_k),\ |\rho-r_k/\lambda_k|\leq k &  \hbox{ in Case A},\\
   \rho\in[-r_k/\lambda_k,(R-r_k)/\lambda_k),\ |\rho|\leq k & \hbox{ in Case B},\\
   \rho\in[0,R/\lambda_k),\ |(R-r_k)/\lambda_k-\rho|\leq k & \hbox{ in Case C}.
\end{cases} $$
Consequently, a subsequence of $(\tilde u_k,\tilde v_k)$ (still denoted
$(\tilde u_k,\tilde v_k)$) converges locally uniformly to a nontrivial solution $(\tilde u,\tilde v)\in{\mathcal K}$
of problem \eqref{FstatL} or \eqref{FstatL} with $n=1$ or \eqref{FstatL-half} in Case~A or~B or~C,
respectively (notice that $(\tilde u,\tilde v)$ is radial in Case~A).

In Case C or if $\beta\ne-1$, then we obtain a contradiction with Theorem~\ref{thmSchrL}.

Assume $\beta=-1$ and consider Case A or B. 
Then Theorem~\ref{thmSchrL} and $M(\tilde u_k,\tilde v_k)(\tilde\rho_k)=1$ guarantee 
$(\tilde u,\tilde v)=(c,\pm c)$, where $c=2^{-2/(p-1)}$.
Replacing $v_k$ by $-v_k$ (and $C_3$ by $C_4$) if necessary, we may assume $(\tilde u,\tilde v)=(c,c)$.
Since $(u_k,v_k)\in{\mathcal K}^*$, we have
$\tilde w_k:=\tilde u_k-\tilde v_k\not\equiv0$ and we also have
$$ 
\begin{aligned}
 & \qquad \tilde w_k''+P_k\tilde w_k'+Q_k\tilde w_k=0, \quad\hbox{where} \\
 &\  P_k:=\frac{n-1}{\rho+\rho_k}, \quad
 Q_k:= \lambda_k^2(\gamma-\lambda) 
 + \frac{|\tilde u_k|^{p-1}\tilde u_k-|\tilde v_k|^{p-1}\tilde v_k}{\tilde u_k-\tilde v_k}
 -\beta_k|\tilde u_k\tilde v_k|^{q}\tilde u_k\tilde v_k.
\end{aligned}
$$
Notice also that $\frac12 P_k'+\frac14 P_k^2=\frac{(n-3)(n-1)}{4(\rho+\rho_k)^2}$.
Fix $R_1>(p-1)^{-1/2}$ and consider $R_2>R_1$ and $\rho\in(R_1,R_2)$.
Since $\beta_k\to-1$ and $\tilde u_k,\tilde v_k\to c$ locally uniformly, we see
that 
$$ \begin{aligned}
q_k: &=Q_k-\frac12 P_k'-\frac14 P_k^2\geq Q_k-\frac1{4\rho^2}\\
 &\to c^{p-1}(p-\beta)-\frac1{4\rho^2}=\frac14(p+1)-\frac1{4\rho^2}>\frac12,
  \end{aligned}$$
where the convergence is uniform for $\rho\in(R_1,R_2)$.
Set $W_k(\rho)=\tilde w_k(\rho) \exp\bigl(\frac12\int_1^\rho P_k\bigr)$.
Then $W_k''+q_kW_k=0$ and $q_k>1/2$ on $(R_1,R_2)$ for $k$ large enough.
Since the solution $W(r)=\sin(\frac{r}{\sqrt2})$ of the equation $W''+\frac12W=0$ has at least $C_3+2$ zeroes in $(R_1,R_2)$
for $R_2$ large enough, the Sturm comparison theorem guarantees that $z(\tilde w_k)=z(W_k)>C_3$ which contradicts
$(u_k,v_k)\in{\mathcal K}^*$ and concludes the proof. 
\end{proof}

\begin{proof}[Proof of Proposition~\ref{propCL}] 
By $C$ we denote various constants which depend only on $\|{\mathcal U}_0\|$ and $C_E$.

Problem \eqref{Schr-lambda} is well posed in $H^1$, hence
there exists $\delta=\delta(\|{\mathcal U}_0\|)\in(0,1)$ such that
\begin{equation} \label{est-delta}
\|{\mathcal U}(\cdot,t)\|\leq C\ \hbox{ for }\ t\in(0,\delta].
\end{equation}
If $\Omega=B_R$, then this estimate and Corollary~\ref{corS} implies
\begin{equation} \label{L2est}
\int_\Omega|{\mathcal U}|^2(x,t)\,dx\leq C, \qquad t\geq0.
\end{equation}

Multiplying the first and the second equation in \eqref{Schr-lambda} by $u$ and $v$, respectively,
integrating by parts, summing the identities and using $\gamma\leq0$ we obtain
\begin{equation} \label{dU2}
 \frac12\frac{d}{dt}\int_\Omega |{\mathcal U}(x,t)|^2\,dx 
 \geq -(p+1)E(t)+\frac{p-1}2\int_\Omega(|\nabla {\mathcal U}(x,t)|^2+(\lambda+\gamma)|{\mathcal U}(x,t)|^2)\,dx.
\end{equation}
We also have 
\begin{equation} \label{Ut2}
 C\geq E(t_1)-E(t_2)=\int_{t_1}^{t_2}\int_\Omega |{\mathcal U}_t|^2\,dx\,dt, \quad t_2>t_1.
\end{equation}
Set 
$$\tilde\lambda:=\begin{cases} \lambda &\hbox{if $\Omega=B_R$},\\
                               \lambda+\gamma &\hbox{if $\Omega=\R^n$},
 \end{cases}$$
and notice that $\tilde\lambda>0$.
Now \eqref{dU2}, \eqref{L2est} and the boundedness of $E$, and then the Cauchy
inequality and \eqref{Ut2} guarantee
$$ \begin{aligned}
 \int_{t}^{t+1}\int_\Omega(|\nabla {\mathcal U}|^2+\tilde\lambda|{\mathcal U}|^2)\,dx\,dt
 &\leq C\Bigl(1+\int_{t}^{t+1}\int_\Omega|{\mathcal U}|\cdot|{\mathcal U}_t|\,dx\,dt\Bigr) \\
 &\leq C\Bigl(1+\Bigl(\int_{t}^{t+1}\int_\Omega|{\mathcal U}|^2\,dx\,dt\Bigr)^{1/2}\Bigr),
\end{aligned}$$
which first shows $\int_{t}^{t+1}\int_\Omega|{\mathcal U}|^2\,dx\,dt\leq C$,
and then 
\begin{equation} \label{CL}
\int_{t}^{t+1}\|{\mathcal U}(\cdot,s)\|^2\,ds\leq C.
\end{equation}

Since ${\mathcal U}$ solves the linear equation ${\mathcal U}_t=\Delta {\mathcal U}-\lambda {\mathcal U}+H{\mathcal U}$, 
where the matrix $H=H(x,t)$
satisfies $\|H(\cdot,t)\|_\infty\leq C$ for any $t\geq\delta$ due to Corollary~\ref{corS},
we have 
\begin{equation} \label{est2}
\|{\mathcal U}(\cdot,t_0+\tau)\|\leq C(\|{\mathcal U}(\cdot,t_0)\|)\ \hbox{ whenever }\ t_0\geq \delta,\ \tau\in[0,2].
\end{equation}
Choosing $t_0=\delta$ in \eqref{est2} and using \eqref{est-delta} we obtain $\|{\mathcal U}(\cdot,t)\|\leq C$ for $t\in[0,2]$. 
Next \eqref{CL} guarantees that for each $k=2,3,\dots$ we can find $t_k\in[k-1,k]$ such that
$\|{\mathcal U}(\cdot,t_k)\|\leq C$ and \eqref{est2} guarantees $\|{\mathcal U}(\cdot,t)\|\leq C$ for $t\in[k,k+1]$.
This concludes the proof.
\end{proof}

\begin{proof}[Proof of Proposition~\ref{prop-comp}] 
If $\Omega=B_R$, then the statement follows from the continuity and boundedness of the trajectory, 
and the smoothing properties of the semiflow generated by \eqref{Schr-lambda}. 
In fact, standard estimates based on the variation of constant formula guarantee that
${\mathcal U}(\cdot,t)$ is bounded in $H^2(B_R,\R^2)$ for $t\geq\delta$, hence the compactness follows from
the compact embedding of $H^2(B_R,\R^2)$ into $H^1(B_R,\R^2)$.

Next let $\Omega=\R^n$, ${\mathcal U}_0$ be compactly supported and $n\geq2$.
It is well known (see \cite{S,L}, for example), that 
$H^1_r=H^1_r(\R^n,\R^2)$ is compactly embedded into $L^s$ if $2<s<p_S$.
It is also easily seen that the function $M(r):=\delta e^{-\eps (r-R)}$, $r>R$,
is a supersolution to problem \eqref{Schr-lambda} for any $R>0$ if $\eps,\delta>0$ are small enough
(where the smallness depends only on $\lambda$ and $\sup_{|{\mathcal U}|=1}|{\mathcal F}({\mathcal U})|$). 
More precisely, if $|{\mathcal U}_0(r)|\leq M(r)$ for $r>R$ and $|{\mathcal U}(R,t)|<M(R)$ for all $t\geq0$, then
$|{\mathcal U}(r,t)|\leq M(r)$ for all $r\geq R$ and $t\geq0$.
Fix such $\eps,\delta$.

Since \cite[Radial Lemma]{S} guarantees 
$|{\mathcal U}(x,t)|\leq C(n)|x|^{(1-n)/2}\|{\mathcal U}(\cdot,t)\|$
and ${\mathcal U}_0$ is compactly supported, we can find $R>0$ such that the support of ${\mathcal U}_0$
is contained in $B_R$ and $|{\mathcal U}(R,t)|<\delta$ for all $t$.
Consequently, we obtain $|{\mathcal U}(r,t)|\leq M(r)$ for all $r\geq R$ and $t\geq0$,
hence the trajectory of ${\mathcal U}$ is bounded in $L^1$.
This fact and the compactness in $L^s$ guarantee the compactness in $L^2$,
and smoothing arguments also prove the compactness in $H^1$. 
In fact, due to Corollary~\ref{corS} one can easily show that
the mapping $L^2\to H^1:{\mathcal U}(\cdot,t)\to {\mathcal U}(\cdot,t+1)$ is continuous.
\end{proof}


\smallskip
{\bf Acknowledgements.} The author thanks Zhi-Qiang Wang for
his helpful comments on paper \cite{LW}.


\end{document}